\documentclass[12pt]{amsart}
\usepackage{amsmath,amssymb,amsthm,amscd}

\def\p{\partial}
\def\R{\mathbb{R}}
\def\C{\mathbb{C}}
\def\N{\mathbb{N}}

\def\ve{\varepsilon}

\def\l{\lambda}
\def\L{\Lambda}

\def\i{\sqrt{-1}}

\def\t{\triangle}

\numberwithin{equation}{section}

\newtheorem{prop}{Proposition}[section]
\newtheorem{theo}[prop]{Theorem}
\newtheorem{lemma}[prop]{Lemma}

\newtheorem{rmk}[prop]{Remark}

\begin{document}
\title[The Complex Monge-Amp\`ere equation]
{The Complex Monge-Amp\`ere equation on compact K\"ahler manifolds}

\author{Xiuxiong CHEN}

\address{Department of Mathematics, University of Wisconsin, Madison, WI, 53705}
\email{xxchen@math.wisc.edu}

\author{Weiyong HE}

\address{Department of Mathematics, University of Oregon, Eugene, OR, 97403}
\email{whe@uoregon.edu}

\begin{abstract} We consider the  complex Monge-Amp\`ere equation on a compact K\"ahler manifold $(M, g)$  when the right hand side $F$ has rather weak regularity. In particular we prove that  estimate of $\t\phi$ and the gradient estimate hold when $F$ is in  $W^{1, p_0}$ for any $p_0>2n$. As an application, we show that there exists a classical solution in $W^{3, p_0}$ for  the complex Monge-Amp\`ere equation when $F$ is in $W^{1, p_0}$.  \end{abstract}

\thanks{}
\date{}
\maketitle

\section{Introduction}
Let $M$ be a compact K\"ahler manifold of complex dimension $n$ with a smooth K\"ahler metric $g=g_{i\bar j}dz^i\otimes d\bar z^j$. The corresponding K\"ahler form is given by $\omega=\i g_{i\bar j} dz^i\wedge d\bar z^j$. The K\"ahler forms in the class $[\omega]$ can be written in terms of a K\"ahler potential $\omega_\phi=\omega+\i \p\bar \p \phi$. 
We shall use the following notations, for a function $f$ and  a holomorphic coordinate $z=(z^1, \cdots, z^n)$,
\[
f_{i\bar j}=\frac{\p^2 f}{\p z^i\p \bar z^j}, \t f=g^{i\bar j} f_{i\bar j}, \t_\phi f=g^{i\bar j}_\phi f_{i\bar j}. 
\]
It is well known that the Ricci curvature of $g$ is given by
\[
R_{i\bar j}=-\frac{\p^2}{\p z^i\p \bar z^j}\log(\det(g_{i\bar j})). 
\]
In particular, the Ricci form 
\[
\rho=\frac{\i}{2\pi}R_{i\bar j} dz^i\wedge d\bar z^j
\] is a close $(1, 1)$ form and it defines a cohomology class independent of the choice of the K\"ahler metric $g$, which is the first Chern class $c_1$ of $M$. 
Therefore if a closed $(1, 1)$ form is the Ricci form of some K\"ahler metric, then its cohomology class must represent the first Chern class $c_1$. In the 1950s, Calabi \cite{Calabi53} conjectured that any form in $c_1$ can actually be written as the Ricci form of some K\"ahler metric.
The Calabi conjecture can be reduced to solving the following complex Monge-Amp\`ere equation
\begin{equation}\label{E-1-1}
\log\left(\frac{\det(g_{i\bar j}+\phi_{i\bar j})}{\det(g_{i\bar j})}\right)=F,
\end{equation}
where the function $F$ satisfies
\[
\int_M\exp(F)dvol_g =Vol(M).
\]
When $F$ is a smooth function on $M$, Yau \cite{Yau} proved that \eqref{E-1-1} has a smooth solution, hence proved the Calabi conjecture; his solution also provided K\"ahler-Ricci flat metrics on K\"ahler manifolds with zero first Chern class. More generally, Calabi  initiated the study of K\"ahler-Einstein metrics, with
\[
\rho_\phi=\l \omega_\phi
\]
for constant $\l$. The K\"ahler-Einstein equation can be deduced to
\begin{equation}\label{E-1}
\log\left(\frac{\det(g_{i\bar j}+\phi_{i\bar j})}{\det(g_{i\bar j})}\right)=F-\l \phi.
\end{equation}
When $\l<0$, Aubin \cite{Aubin} and  Yau \cite{Yau} have proved the existence of a smooth solution of \eqref{E-1}. When $\l$ is positive, and so $M$ is a Fano manifold in algebraic-geometric language, \eqref{E-1} may or may not have a smooth solution. In this case, Tian has made enormous progress towards understanding precisely when a solution exists , see  \cite{Tian90, Tian97} for example. In particular, a complete answer \cite{Tian90} for  existence of K\"ahler-Einstein was given when $n=2$;  see also recently Chen-Wang's proof  \cite{Chen-Wang1, Chen-Wang2}  via K\"ahler-Ricci flow. 

Weak solution in various settings of complex Monge-Ampere equations has also been studied extensively since the pioneering work of Bedford-Taylor \cite{bt01, bt02}. Kolodziej \cite{K98, K1} proved that 
there exists a unique bounded solution of 
\begin{equation}\label{E-weak}
\omega_\phi^n=\tilde F \omega^n
\end{equation}
when $\tilde F$ is a nonnegative $L^p$ function for $p>1$; moreover, the solution is H\"older continuous.  There are further  regularity, existence and uniqueness results on \eqref{E-weak} when $\tilde F$ is less regular and/or when $\omega$ is degenerate, to mention  \cite{Zhang, Blocki05, Guedj-Zeriahi, Demailly-Pali, Eyssidieux-Guedj-Zeriahi, Dinew} to name a few. Readers are referred to the aforementioned references for a historic overview and further references.

In the early 1980s, Calabi initiated another problem \cite{Calabi82} in K\"ahler geometry to seek extremal K\"ahler metrics, which include metrics of constant scalar curvature as a special case,
\begin{equation}\label{E-2}
R_\phi=-g^{i\bar j}_\phi\p_i\p_{\bar j}\log\left(
\det(g_{k\bar l}+\phi_{k\bar l})\right)=\underline{R},
\end{equation}
where the constant $\underline{R}$ is determined by $(M, [\omega])$. There is a tremendous body of work on the problems related to extremal metrics and much progress has been made in the last decade. The well known conjecture of Yau-Tian-Donaldson  \cite{Yau90, Tian97, Donaldson04} asserts  that if  $(M, [\omega])$ is a compact K\"ahler manifold and $[\omega]=2\pi c_1(L)$ for some holomorphic line bundle $L\rightarrow M$,  then there is a metric of constant scalar curvature in the class $[\omega]$ if and only if $(M, L)$ is K-stable. This is a core problem in K\"ahler geometry. Thorough the efforts of many mathematicians, the necessary part of this conjecture is essentially done (\cite {Tian97, Donaldson02, Donaldson051, Chen-Tian, Mabuchi04, Stoppa, ss}  ). The core problem is to find a way to understand existence problem.

Back to 1980s,   Calabi  proposed a parabolic equation, the Calabi flow, 
\begin{equation}\label{E-3}
\frac{\p \phi}{\p t}=R_\phi-\underline{R}. 
\end{equation}
One of the biggest difficulties to understand extremal metrics, in particular for equations \eqref{E-2} and \eqref{E-3}, is how to control the metric through its scalar curvature, in contrast to its Ricci curvature. For example,  it was shown \cite{Chen-He} that all metrics in $[\omega]$ are equivalent and pre-compact in $C^{1, \alpha}$ topology when Ricci curvature and potential are both uniformly bounded; 
which is used to prove the extension of the Calabi flow with Ricci curvature bound. But such a result for scalar curvature is not known. 
Recently, Donaldson solved the constant scalar curvature equation on toric K\"ahler surfaces with K-stability in a series of paper \cite{Donaldson02, Donaldson05, Donaldson051, Donaldson08, Donaldson09}. 
The scalar curvature equation in this case  can be formulated \cite{Abreu} as a fourth order equation for convex functions on certain convex polytope in $\R^2$. Roughly speaking, the key point in Donaldson's work is  to control the metric through its scalar curvature and K-stability condition; in this case,   the metrics can be written in terms of the hessian of convex functions and convex analysis plays an important role.

While this progress is indeed impressive,  the general existence problem is very elusive 
and beyond us.  In retrospect, the memorable feature of Yau's solution to Calabi conjecture
is as follows:  given a K\"ahler manifold with null first Chern class, does the $C^{2}$ estimate of potential
follow from $C^{0}$ estimate of potential immediately? \\

{\bf Guiding Problem:} Given the cscK equation,  can we control $C^{2}$ and higher derivative
estimates in terms of $C^{0}$ and/or $C^{1}$ estimates?  Is this realistic?\\

 Donaldson's work is very inspirational in the sense that this is true for cscK metrics on toric K\"ahler surfaces. In other words, for cscK metric on a toric surface, once K\"ahler potential is bounded, then all higher derivatives will be bounded {\it a priori.}  One of course would  expect the more challenge case 
 for general K\"ahler manifolds. Nevertheless this is a key problem in K\"ahler geometry we wish
  to consider.

We can view the scalar curvature equations as two coupled second order equations:
\[
\det(g_{i\bar j}+\phi_{i\bar j}) = e^{F} \det g,
\]
and
\[
R_\phi=- \t_\phi F - g^{i\bar j}_{\phi} (R_{i\bar j}(g)).
\]
We hope to get some weak regularity of the volume form through its scalar curvature (the second equation), and then try to control the metric/potential  by its volume form from the first equation. To make this strategy successful,  we want to impose as weak regularity on $F$ as possible, but still try to get a upper bound control on $n+ \triangle \phi$, hence control the metric.\\

Motivated in part by these problems, we are interested in the classical solution of  \eqref{E-1-1} when $F$ has rather weak regularity; in particular, we assume  $F\in W^{1, p_0}$ for $p_0>2n$. We believe that such results would be important, in particular for problems related to scalar curvature in K\"ahler geometry.\\

One of the key steps towards solving \eqref{E-1-1} is the estimate of $\t \phi$, which is obtained  by Aubin \cite{Aubin} and Yau \cite{Yau}.  Roughly speaking, such an estimate asserts that, for a smooth solution of \eqref{E-1-1}, 
\begin{equation}\label{E-1-2}
0<n+\t \phi\leq C e^{C_1(\phi-\inf \phi)},
\end{equation}
where $C$ depends on $F$ up to its second derivatives. 
To deal with the case when $F$ has weaker regularity, we use integral method and Moser's iteration (see \cite{moser}) instead of the maximum principle.   This allows us to obtain estimate of $\t \phi$ and gradient estimate of $\phi$ when $F\in W^{1, p_0}$ for $p_0>2n$.  
In particular, we have
\begin{theo}\label{T}
Let $M$ be a compact K\"ahler manifold of complex dimension $n$ with a smooth metric $g=g_{i\bar j}dz^i\otimes d\bar z^j$. 
Let $F$ be a function in $W^{1, p_0}$ for some $p_0>2n$, then  \eqref{E-1-1} has a classical solution $\phi\in W^{3, p_0}$. 
\end{theo}

To prove Theorem \ref{T}, we shall derive the following  two {\it a priori} estimates. First we have the estimate of $\t \phi$, which depends in addition on the Lipschitz norm of $\phi$. 
 \begin{theo}\label{T-0} If $\phi$ is a smooth solution of \eqref{E-1-1},  then
 \begin{equation}\label{E-4}
 0<n+\t \phi\leq C=C(\|\phi\|_{L^\infty}, \|\nabla \phi\|_{L^\infty}, \|F\|_{W^{1, p_0}}, p_0, M, g, n).
 \end{equation}
 \end{theo}

\begin{rmk}
Note that in Theorem \ref{T-0}, $p_0>2n$ is optimal in the sense that when $p_0\leq 2n$, $F\in W^{1, p_0}(M, g)$ does not imply $F\in L^\infty$ by the Sobolev embedding. Actually in general this is not true; hence \eqref{E-4} cannot hold in general when $F\in W^{1, p_0}$ for $p_0\leq 2n$. In particular, when $p_0\leq 2n$, one cannot expect a classical solution in Theorem \ref{T}. 
\end{rmk}

The  $L^\infty$ estimate of $\phi$ was originally derived by Yau \cite{Yau} after obtaining estimate of $\t\phi$ in \eqref{E-1-2}.   Kolodziej \cite{K98} proved that  there exists a unique bounded solution $\phi$ of \eqref{E-weak}  $\tilde F\in L^p$ and $\tilde F$ is nonnegative, $p>1$; note that his result  holds  for the degenerate case. Later on he proved the H\"older estimate of $\phi$ in \cite{K1}.  
Hence we shall assume that $\|\phi\|_{L^\infty}$ is bounded and derive further regularity of $\phi$. 

The gradient estimate of $\phi$ was not required in \cite{Aubin, Yau} to derive the estimate of $\t \phi$. But we shall need this estimate.  Such an estimate, which also holds for \eqref{E-weak} and the Dirichlet problem of the complex Monge-Amp\`ere equation on manifolds with boundary,  is derived by \cite{Hanani, Chen, Blocki1, Guan, Guan-Li, Phong-Sturm} in various settings;  in particular, Blocki \cite{Blocki1} established the gradient estimate of \eqref{E-1-1} when  $F$ is  Lipschitz (we should mention that his results also holds for $\eqref{E-weak}$ when). We shall prove, when $F$ in $W^{1, p_0}$, $p_0>2n$, that
\begin{theo}\label{T-1}If $\phi$ is a smooth solution of \eqref{E-1-1},  then
 \begin{equation}\label{E-5}
 |\nabla \phi| \leq C(\|\phi\|_{L^\infty}, \|F\|_{W^{1, p_0}}, p_0, M, g, n).
 \end{equation}
 \end{theo}
 
With a slight modification, the proof of  Theorem \ref{T-0} and Theorem \ref{T-1} holds for $\eqref{E-1}$, see Remark \ref{R-1} and Remark \ref{R-2}. In particular, the $L^\infty$ bound of $\phi$ follows from the maximum principle when $\l=-1$; hence Theorem \ref{T} also holds for \eqref{E-1} when $\l=-1$. 
 
We still use integral method and iteration argument to prove Theorem \ref{T-1};  the new ingredient is that we obtain some inequalities, see \eqref{E-2-11} and \eqref{E-2-15} below, which are essential for the iteration process. Similar inequalities are derived  in \cite{Blocki1},  but mainly with an emphasis on the point of local maximum of the barrier functions, which is sufficient for the maximum principle argument, but not for our case.

The integral method uses the ideas in the well-known theory of De Giorgi \cite {degiorgi}, Nash \cite{nash} and Moser \cite{moser}. Roughly stated, such a theory deals with the regularity problem of  the elliptic operator of the form
\[
L u= D_i(a^{ij}D_j u)=f
\]  
with measurable coefficients $a^{ij}$ such that $\l |\xi|^2\leq a^{ij}\xi_i\xi_j\leq \L |\xi|^2$ for some positive constants $\l, \L$.
Uniform ellipticity is essential in De Giorgi-Nash-Moser theory.

One observation in the present paper is that the upper bound for $a^{ij}$ is not necessary on compact manifolds (without boundary), which is required essentially to deal with the terms from cut-off functions.  To overcome the absence of  the {\it apriori} lower bound for $g^{i\bar j}_\phi$ (or the absence of the {\it apriori} uniform Sobolev constant for $\omega_\phi$),  we use $a^{ij}=(n+\t\phi)g^{i\bar j}_\phi$ instead when estimating $\t\phi$. For example, to estimate $\t\phi$,  we compute 
\begin{equation}\label{E-5}
\int_M ug^{i\bar j}_\phi (u^p)_i (u^p)_{\bar j} dvol_\phi,
\end{equation}
where $u=\exp(-C\phi)(n+\t \phi)$ is the barrier function used in \cite{Yau}.
Combining Yau's computation \cite{Yau}, this  gives the starting point of iteration process. The iteration process for the gradient estimate is much more complicated than \eqref{E-5} technically, which depends on the computations as in \eqref{E-2-11}--\eqref{E-2-20} in an essential way.  But module  these technical details, we believe that these ideas can be applied to other nonlinear equations on compact manifolds. 

\vspace{2mm}

{\bf Acknowledgements:} 
The authors are partially supported by an NSF grant.  WYH is  grateful to Prof. J.Y. Chen for constant support and encouragement. 

\section {Estimates of $\t\phi$}
We shall always assume a normalization condition
\begin{equation}\label{E-1-3}
\int_M \phi dvol_g=0.
\end{equation}
We shall also assume that $\phi$ and $F$ are both smooth and derive a priori estimates of $\t \phi$ depending on $\|\phi\|_{L^\infty}, \|\nabla \phi\|_{L^\infty}, \|F\|_{W^{1, p_0}}, p_0, M, g, n$. For simplicity, we shall not emphasize the dependence on $M, g$; while the dependence on the geometry of $(M, g)$ can be made quite explicit from the computation. 
We shall then  prove Theorem \ref{T-0}. 
\begin{proof}
We start with Yau's computation \cite{Yau}, 
\begin{equation}\label{E-1-4}\begin{split}
&\t_\phi\left (\exp(-C_1\phi)(n+\t \phi)\right)\\&\quad\geq \exp(-C_1\phi)\left[\t F-n^2 \inf_{i\neq l}R_{i\bar i l\bar l}-C_1n(n+\t \phi)\right]\\
&\quad\quad+\exp\left(-C_1\phi-\frac{F}{n-1}\right)\left(C_1+\inf_{i\neq l}R_{i\bar il\bar l}\right)(n+\t \phi)^{\frac{n}{n-1}}.
\end{split}
\end{equation}
Note that for any $\ve>0$, by Young's inequality,  there is a constant $C=C(\ve, n)$ such that
\[
n+\t \phi\leq \ve (n+\t \phi)^{\frac{n}{n-1}}+C(\ve, n).
\]
Denote $u=\exp(-C_1\phi)(n+\t \phi)$.
We can rewrite \eqref{E-1-4} as
\begin{equation}\label{E-1-5}
C_2 (n+\t \phi)^{\frac{n}{n-1}}+\exp(-C_1\phi)\t F-C_3\leq \t_\phi u,
\end{equation}
where  $C_2, C_3$ depend on $C_1, \|\phi\|_{L^\infty}, \sup F, n$. Note that $u>0$ is bounded from below by a fixed positive number.  
We use $dvol_g$ to denote the volume form of $g$, and $dvol_\phi$ the volume form of $g_{i\bar j}+\phi_{i\bar j}$. In particular,
\[
dvol_\phi=\exp(F) dvol_g.
\]
We compute, for $p>0$, 
\begin{equation}\label{E-1-6}
\t_\phi (u^p)=pu^{p-1}\t_\phi u+p(p-1)u^{p-2}|\nabla u|^2_\phi.
\end{equation}
Integration by parts, we compute
\begin{equation}\label{E-1-7}
\int_M u^{p+1}\t_\phi(u^p) dvol_\phi=-\frac{p+1}{p}\int_M u|\nabla (u^p)|^2_\phi dvol_\phi.
\end{equation}
On the other hand by \eqref{E-1-6}, we compute
\begin{equation}\label{E-1-8}\begin{split}
\int_M u^{p+1}\t_\phi(u^p) dvol_\phi&=p\int_Mu^{2p}(\t_\phi u) dvol_\phi+p(p-1)\int_Mu^{2p-1}|\nabla u|^2_\phi dvol_\phi\\
&=p\int_Mu^{2p}(\t_\phi u) dvol_\phi+\frac{p-1}{p}\int_M u|\nabla(u^p)|^2_\phi dvol_\phi.
\end{split}
\end{equation}
It follows from  \eqref{E-1-7} and \eqref{E-1-8} that
\begin{equation}\label{E-1-9}
-2\int_M u|\nabla (u^p)|^2_\phi dvol_\phi=p\int_Mu^{2p}(\t_\phi u) dvol_\phi.
\end{equation}
We then compute, by \eqref{E-1-5},
\begin{equation}\label{E-1-10}
\int_Mu^{2p}(\t_\phi u) dvol_\phi\geq \int_M u^{2p}\left(C_2(n+\t \phi)^{\frac{n}{n-1}}+\exp(-C_1\phi)\t F-C_3\right)dvol_\phi.
\end{equation}
To deal with the term involved with $\t F$, we compute, 
\begin{equation}\label{E-1-11}
\begin{split}
&\int_M u^{2p}\exp(-C_1\phi) \t F  dvol_\phi\\
&=\int_M u^{2p}\exp(-C_1\phi) \t F \exp(F) dvol_g\\
&=-\int_M \nabla F \nabla \left(\exp(F-C_1\phi)u^{2p}\right)dvol_g\\
&=-\int_M\left( \left|\nabla F\right|^2 u^{2p}-C_1\nabla F \nabla \phi u^{2p}+2\nabla(u^p) \nabla Fu^p \right)\exp(F-C_1\phi) dvol_g\\
&=-\int_M \left( \left|\nabla F\right|^2 u^{2p}-C_1\nabla F \nabla \phi u^{2p}+2\nabla(u^p) \nabla Fu^p \right) \exp(-C_1\phi)dvol_\phi.
\end{split}
\end{equation}
By \eqref{E-1-9}, \eqref{E-1-10} and \eqref{E-1-11}, we compute
\begin{equation}\label{E-1-12}\begin{split}
&\int_M u|\nabla (u^p)|^2_\phi dvol_\phi\leq \frac{p}{2}\int_M u^{2p}\left(C_3-C_2(n+\t \phi)^{\frac{n}{n-1}}\right) dvol_\phi\\
&\quad+\frac{p}{2} \int_M \left( \left|\nabla F\right|^2 u^{2p}-C_1\nabla F \nabla \phi u^{2p}+2\nabla(u^p) \nabla Fu^p \right) \exp(-C_1\phi)dvol_\phi\\
&\leq pC_4\int_M (1+|\nabla F|^2+|\nabla F||\nabla \phi|)u^{2p} dvol_\phi\\
&\quad+pC_4\int_{M}|\nabla F| |\nabla (u^p)| u^{p}dvol_\phi-\frac{p}{2}C_2\int_M u^{2p}(n+\t \phi)^{\frac{n}{n-1}} dvol_\phi,
\end{split}
\end{equation}
where $C_4=C_4(C_1, C_3, \|\phi\|_{L^\infty}). $
Note that $(n+\t \phi)g^{i\bar j}_\phi \xi_i\bar \xi_j\geq |\xi|^2$ for any vector $\xi=(\xi_1, \cdots, \xi_n)$. Then we compute
\begin{equation}\label{E-1-13}
\begin{split}
\int_M u|\nabla (u^p)|^2_\phi dvol_\phi=&\int_M \exp(-C_1\phi) (n+\t \phi)g^{i\bar j}_\phi \p_i (u^p)\p_{\bar j}(u^p) dvol_\phi\\
\geq&a\int_M |\nabla (u^p)|^2 \exp(F)dvol_g,
\end{split}
\end{equation}
where $a$ is a constant depending only on $C_1, \|\phi\|_{L^\infty}$,  and 
\begin{equation}\label{E-1-14}\begin{split}
pC_4\int_M\left|\nabla (u^p)\right|\left|\nabla F\right|u^p dvol_\phi\leq& \frac{a}{2} \int_M  |\nabla (u^p)|^2 \exp(F)dvol_g\\
&+\frac{p^2C_4^2}{2a}\int_M |\nabla F|^2 u^{2p}\exp(F)dvol_g.
\end{split}
\end{equation}
Combine \eqref{E-1-12}, \eqref{E-1-13} and \eqref{E-1-14}, we compute
\begin{equation}\label{E-1-15}
\begin{split}
&\int_M\left(|\nabla (u^p)|^2+pC_5 u^{2p+\frac{n}{n-1}}\right) dvol_\phi\\
&\leq p^2 C_6\int_M u^{2p}|\nabla F|^2  dvol_\phi+pC_6\int_M (1+|\nabla F|^2+|\nabla F||\nabla \phi|) u^{2p}dvol_\phi,
\end{split}
\end{equation}
where both positive constants $C_5$, $C_6$ depend on $\|\phi\|_{L^\infty}, \sup F, n$.   We shall assume $p\geq 1/4$ from now on. Then  we can get that, by \eqref{E-1-15}, 
\begin{equation}\label{E-1-16}
\int_M\left(|\nabla (u^p)|^2+pC_7 u^{2p+\frac{n}{n-1}}\right)dvol_g\leq p^2 C_8\int_M u^{2p} \left(|\nabla F|^2+1\right)dvol_g,
\end{equation}
where $C_7=C_7\left(\|\phi\|_{L^\infty}, \|F\|_{L^\infty}, n\right)$ and $C_8=C_8\left(\|\phi\|_{L^\infty}, \|\nabla \phi\|_{L^\infty},  \|F\|_{L^\infty}\right)$.
To get $L^\infty$ bound of $u$, we use the iteration method (see \cite{moser}). Recall the Sobolev inequality for $(M, g)$; there is a constant $C_s=C_s(n, g)$ such that, 
\begin{equation}\label{E-1-19}
\left(\int_M f^{\frac{2n}{n-1}}dvol_g\right)^{\frac{n-1}{n}}\leq C_s \left(\int_M |\nabla f|^2 dvol_g+Vol^{-n}(M, g)\int_M f^2dvol_g\right).
\end{equation}
Note that the above Sobolev inequality  is scaling invariant. 
Let $f=u^p$; it follows from \eqref{E-1-16} and \eqref{E-1-19} that
\begin{equation}\label{E-1-20}
\|u\|_{L^{\frac{2pn}{n-1}}}\leq (pC)^{1/p} \left(\int_M u^{2p}(|\nabla F|^2+1) dvol_g\right)^{\frac{1}{2p}}.
\end{equation}
By the H\"older inequality, we get that
\begin{equation*}
\int_M u^{2p}(|\nabla F|^2+1) dvol_g\leq \left(\int_M u^{2pq_0}\right)^{1/q_0} \left(\int_M (|\nabla F|^2+1)^{\frac{p_0}{2}}\right)^{\frac{2}{p_0}},
\end{equation*}
where $1/q_0+2/p_0=1$. When $|\nabla F|\in L^{p_0}$, it then follows that
\begin{equation}\label{E-1-21}
\|u\|_{L^{\frac{2pn}{n-1}}}\leq (pC)^{1/p} \|u\|_{L^{2pq_0}},
\end{equation}
where $C=C(\|\phi\|_{L^\infty}, \|\nabla \phi\|_{L^\infty}, \|F \|_{W^{1, p_0}})$. When $1<q_0<\frac{n}{n-1}$, let \[
b=\frac{n}{(n-1)q_0}>1.
\]
We can rewrite \eqref{E-1-21} as, for $p \geq 1$, 
\begin{equation}\label{E-1-22}
\|u\|_{L^{pb}}\leq (pC)^{\frac{2q_0}{p}}\|u\|_{L^p}.
\end{equation}
We can  then apply the standard iteration argument; 
let   $p=b^k$ for $k\geq 0$ in \eqref{E-1-22}, 
then we  compute from \eqref{E-1-22} that, for $k\geq 1$,
\[
\log \|u\|_{L^{b^k}}\leq  \frac{2q_0}{b^{k-1}}\left(\log b^{k-1}+\log C\right)+\log \|u\|_{L^{b^{k-1}}}.
\]
It follows that
\[
\log \|u\|_{L^{b^{k+1}}}\leq \sum_{i=0}^k \frac{2q_0}{b^i}\left(\log b^i+\log C\right)+\log \|u\|_{L^1}. 
\]
Since $b>1$, it is clear that, when $k\rightarrow\infty$, 
\[
 \sum_{i=0}^\infty \frac{2q_0}{b^i}\left(\log b^i+\log C\right)\leq C.
\]
It then follows that, when $k\rightarrow \infty$,
\begin{equation}\label{2-22}
\|u\|_{L^\infty}\leq C\|u\|_{L^1}\leq C,
\end{equation}
where $C=C(\|\phi\|_{L^\infty}, \|\nabla \phi\|_{L^\infty}, \|F \|_{W^{1, p_0}}, p_0, M, g,  n)$.  In other words, 
\begin{equation}\label{E-1-23}
0<n+\t\phi\leq C.
\end{equation}
\end{proof}

\begin{rmk}When $(M, g)$ has nonnegative bisectional curvature, one can get 
\[
\t_\phi (\t\phi)\geq \t F.
\]
By taking $u=n+\t\phi$ as in the arguments above, one can derive the following estimate of $\t\phi$ directly, without even assuming $\|\phi\|_{L^\infty}$ bound,
\[
0<n+\t\phi\leq C(\|F\|_{W^{1, p_0}}, p_0, M, g, n). 
\]
\end{rmk}

\begin{rmk}\label{R-1}For simplicity let $\l=1$ or $-1$ in \eqref{E-1}. Then \eqref{E-1-5} still holds (with different $C_2$ and $C_3$). Hence Theorem \ref{T-0} holds for $\eqref{E-1}$ by the same argument.
\end{rmk}

\section {The gradient estimate}

We shall prove Theorem \ref{T-1} in this section. 
\begin{proof}

We shall assume that $F$ and $\phi$ are both smooth and derive estimate of $\|\nabla \phi\|_{L^\infty}$ depending on $\|F\|_{W^{1, p_0}}, p_0, \|\phi\|_{L^\infty}, M, g, n$.  
We compute
\begin{equation}\label{E-2-1}
\begin{split}
\t_\phi (|\nabla \phi|^2)=&g^{i\bar j}_\phi \p_i\p_{\bar j} \left(g^{k\bar l}\phi_k\phi_l\right)\\
=&g^{i\bar j}_\phi \p_i (\p_{\bar j}g^{k\bar l}\phi_k\phi_{\bar l}+g^{k\bar l} \p_{\bar j}\phi_k\phi_{\bar l}+g^{k\bar l}\phi_k\p_{\bar j}\phi_{\bar l})\\
=&g^{i\bar j}_\phi(\p^2_{i\bar j} g^{k\bar l}\phi_k\phi_{\bar l}+\p_{\bar j}g^{k\bar l} \p_i(\phi_k\phi_{\bar l})+\p_i g^{k\bar l}( \p_{\bar j}\phi_k\phi_{\bar l}+\phi_k\p_{\bar j}\phi_{\bar l}))\\
&+g^{i\bar j}_\phi g^{k\bar l}\p_i (\p_{\bar j}\phi_k\phi_{\bar l+}\phi_k\p_{\bar j}\phi_{\bar l}).
\end{split}
\end{equation}
For simplicity, we can pick up a coordinate system such that at one point, $g_{i\bar j}=\delta_{i j}, \p_k g_{i\bar j}=\p_{\bar k}g_{i\bar j}=0$. 
We take derivative of \eqref{E-1-1}, at the given point, 
\begin{equation}\label{E-2-2}
g^{i\bar j}_\phi \p^3_{i,\bar j, k}\phi=F_k; g^{i\bar j}_\phi \p^3_{i,\bar j, \bar k}\phi=F_{\bar k}.
\end{equation}
Then we compute from \eqref{E-2-1}, \eqref{E-2-2}, 
\begin{equation}\label{E-2-3}
\t_\phi (|\nabla \phi|^2)=g^{i\bar j}_\phi g^{k\bar q}g^{p\bar l}R_{i\bar jp\bar q} \phi_k\phi_{\bar l}+g^{k\bar l}(F_k\phi_{\bar l}+F_{\bar l}\phi_k)+g^{i\bar j}_\phi g^{k\bar l}(\phi_{k\bar j}\phi_{i\bar l}+\phi_{ki}\phi_{\bar j\bar l}),
\end{equation}
where $R_{i\bar jp \bar q}$ is the bisectional curvature and  we use the notion of the covariant derivatives
\[
\phi_{ki}=\p^2_{k, i}\phi-\Gamma^l_{ki}\phi_l;  \phi_{\bar j\bar l}=\p^2_{\bar j, \bar l}\phi-\Gamma_{\bar j\bar l}^{\bar q}\phi_{\bar q}.
\]
Let $A(t): \R\rightarrow \R$ be an auxiliary function which will be specified later. We compute

\begin{equation}\label{E-2-4}
\t_\phi A(\phi)=g^{i\bar j}_\phi\p_i (A^{'}\phi_{\bar j})
=g^{i\bar j}_\phi (A^{''}\phi_i\phi_{\bar j}+A^{'}\phi_{i\bar j}),
\end{equation}and 
\begin{equation}\begin{split}\label{E-2-5}
\t_\phi e^{-A(\phi)}&=e^{-A(\phi)}((A^{'})^2-A^{''})g^{i\bar j}_\phi \phi_i\phi_{\bar j}-e^{-A(\phi)}A^{'}g^{i\bar j}_\phi \phi_{i\bar j}\\
&=e^{-A(\phi)}\left((A^{'})^2-A^{''}\right)g^{i\bar j}_\phi \phi_i\phi_{\bar j}+e^{-A(\phi)}A^{'}g^{i\bar j}_\phi g_{i\bar j}-n e^{-A(\phi)}A^{'}.\end{split}
\end{equation}
Then by \eqref{E-2-3} and \eqref{E-2-5}, we compute
\begin{equation}\label{E-2-6}
\begin{split}
\t_\phi \left(e^{-A(\phi)}|\nabla \phi|^2\right)=&\t_\phi e^{-A(\phi)} |\nabla \phi|^2+e^{-A(\phi)}\t_\phi (|\nabla\phi|^2)\\&+g^{i\bar j}_\phi \left(\p_i(e^{-A(\phi)})\p_{\bar j}(|\nabla\phi|^2)
+\p_{\bar j}(|\nabla\phi|^2)\p_i(e^{-A(\phi)})\right)\\
=& e^{-A(\phi)}\left((A^{'})^2-A^{''}\right)g^{i\bar j}_\phi \phi_i\phi_{\bar j}|\nabla \phi|^2+e^{-A(\phi)}A^{'}g^{i\bar j}_\phi g_{i\bar j}|\nabla \phi|^2\\
&-ne^{-A(\phi)}A^{'}|\nabla \phi|^2+e^{-A(\phi)}g^{i\bar j}_\phi g^{k\bar q}g^{p\bar l}R_{i\bar jp\bar q} \phi_k\phi_{\bar l}\\
&+e^{-A(\phi)}g^{k\bar l}(F_k\phi_{\bar l}+F_{\bar l}\phi_k)+e^{-A(\phi)}g^{i\bar j}_\phi g^{k\bar l}(\phi_{k\bar j}\phi_{i\bar l}+\phi_{ki}\phi_{\bar j\bar l})\\
&-e^{-A(\phi)}A^{'}g^{i\bar j}_\phi g^{k\bar l}\left(\phi_i(\phi_{k\bar j}\phi_{\bar l}+\phi_k\phi_{\bar j\bar l})+\phi_{\bar j}(\phi_{ki}\phi_{\bar l}+\phi_k\phi_{i\bar l})\right).
\end{split}
\end{equation}
To estimate the right hand side of \eqref{E-2-6}, we can pick up a coordinate system such that, at the given point, $g_{i\bar j}=\delta_{ij}, \phi_{i\bar j}=\delta_{i j}\phi_{i\bar i}=\l_i$. 
Note that
\[
g^{i\bar j}_\phi g_{i\bar j}=\sum_i \frac{1}{1+\l_i}, g^{i\bar j}_\phi \phi_{i}\phi_{\bar j}=\sum_{i}\frac{\phi_i\phi_{\bar i}}{1+\l_i}.
\]
Then we compute
\begin{equation}\label{E-2-7}
\begin{split}
 &(A^{'})^2g^{i\bar j}_\phi \phi_i\phi_{\bar j}|\nabla \phi|^2+g^{i\bar j}_\phi g^{k\bar l}\phi_{ki}\phi_{\bar j\bar l}-A^{'} g^{i\bar j}_\phi g^{k\bar l}(\phi_i\phi_k\phi_{\bar j\bar l}+\phi_{\bar j}\phi_{\bar l}\phi_{ki})\\
 &=\sum_{k, i} \left((A^{'})^2\frac{\phi_i\phi_{\bar i}}{1+\l_i}\phi_k\phi_{\bar k}+\frac{\phi_{ki}\phi_{\bar k\bar i}}{1+\l_i}-A^{'}\frac{1}{1+\l_i}(\phi_i\phi_k\phi_{\bar k\bar i}+\phi_{\bar i}\phi_{\bar k}\phi_{ki})\right)\\
 &=\sum_{k, i}\frac{1}{1+\l_i}\left(A^{'}\phi_i\phi_k-\phi_{ki}\right)\left(A^{'}\phi_{\bar k}\phi_{\bar i}-\phi_{\bar k\bar i}\right)\geq 0.
\end{split}
\end{equation}
We shall assume that the bisectional curvature of $g$ is bounded from below, namely for some $B\geq 0$, 
\[
R_{i\bar j k\bar l}\geq -B(g_{i\bar j}g_{k\bar l}+g_{i\bar l}g_{k\bar j}).
\]
We can then compute
\begin{equation}\label{E-2-8}
\begin{split}
g^{i\bar j}_\phi g^{k\bar q}g^{p\bar l}R_{i\bar jp\bar q} \phi_k\phi_{\bar l}=&\sum_{k, i, l}\frac{R_{i\bar i k\bar  l}}{1+\l_i}\phi_k\phi_{\bar l}\\
\geq &-B\left(\sum_i \frac{1}{1+\l_i}\right)|\nabla \phi|^2-B\sum_i\frac{\phi_i\phi_{\bar i}}{1+\l_i}. 
\end{split}
\end{equation}
We can also compute
\begin{equation}\label{E-2-9}
g^{i\bar j}_\phi g^{k\bar l}\phi_{k\bar j}\phi_{i\bar l}=\sum_i\frac{\l_i^2}{1+\l_i}=\t \phi-n+\sum_i\frac{1}{1+\l_i},
\end{equation}
and 
\begin{equation}\label{E-2-10}
g^{i\bar j}_\phi g^{k\bar l}(\phi_i\phi_{k\bar j}\phi_{\bar l}+\phi_{\bar j}\phi_k\phi_{i\bar l})=2\sum_i \frac{\l_i\phi_i\phi_{\bar i}}{1+\l_i}=2|\nabla \phi|^2-\sum_i\frac{2\phi_i\phi_{\bar i}}{1+\l_i}.
\end{equation}
We can then estimate \eqref{E-2-6}, taking \eqref{E-2-7}, \eqref{E-2-8}, \eqref{E-2-9} and \eqref{E-2-10} into account, that
\begin{equation}\label{E-2-11}
\begin{split}
\t_\phi\left(e^{-A(\phi)}|\nabla \phi|^2\right)\geq& -A^{''}e^{-A(\phi)}g^{i\bar j}_\phi\phi_i\phi_{\bar j}|\nabla \phi|^2+(A^{'}-B)e^{-A(\phi)}g^{i\bar j}_\phi g_{i\bar j}|\nabla \phi|^2\\
&+(2A^{'}-B)e^{-A(\phi)}g^{i\bar j}_\phi\phi_i\phi_{\bar j}-(n+2)A^{'}e^{-A(\phi)}|\nabla \phi|^2\\
&+e^{-A(\phi)}(\t\phi-n+g^{i\bar j}_\phi g_{i\bar j})-2e^{-A(\phi)}|\nabla F||\nabla\phi|. 
\end{split}
\end{equation}
Now we shall specify the function $A$. Recall that $\|\phi\|_{L^\infty}, \|F\|_{L^\infty}$ are bounded and let $C_0$ be a fixed positive constant such that $C_0=1+ \|\phi\|_{L^\infty}.$ We then choose
\[A(t)=(B+2)t-\frac{t^2}{2C_0}. 
\] 
It then follows that
\[
B+1\leq A^{'}(\phi)=B+2-\frac{\phi}{C_0}\leq B+3, A^{''}(\phi)=-C^{-1}_0.
\]
It is also easy to see that (for example see \cite{Yau}), 
\begin{equation}\label{E-2-12}
g^{i\bar j}_\phi g_{i\bar j}=\sum_i \frac{1}{1+\l_i} \geq (n+\t \phi)^{1/(n-1)}\exp(-F/(n-1)). 
\end{equation}
For simplicity, we then use  $\ve$ to denote a fixed positive small number, and $C$ a fixed positive large number,  if it is not specified, which depend only on $\|\phi\|_{L^\infty}, \|F\|_{L^\infty}, n, B$, $M$, $g$. These constants can vary line by line. 
We then compute, by \eqref{E-2-12},
\begin{equation}\label{E-2-13}\begin{split}
&-A^{''}e^{-A(\phi)}g^{i\bar j}_\phi\phi_i\phi_{\bar j}|\nabla \phi|^2+(A^{'}-B)e^{-A(\phi)}g^{i\bar j}_\phi g_{i\bar j}|\nabla \phi|^2\\
&\quad\geq e^{-A(\phi)}|\nabla \phi|^2\left(\frac{1}{C_0}\sum_i\frac{\phi_i\phi_{\bar i}}{1+\l_i}+\sum_{i}\frac{1}{1+\l_i}\right)\\
&\quad\geq  e^{-A(\phi)}|\nabla \phi|^2\left (\frac{1}{C_0}\sum_i\frac{\phi_i\phi_{\bar i}}{1+\l_i}+(n+\t\phi)^{1/(n-1)}e^{-F/(n-1)}\right)\\
&\quad\geq 2\ve_0|\nabla \phi|^{2+2/n},
\end{split}
\end{equation}
where we have applied the inequality, for each $i$, 
\[
\frac{\phi_i\phi_{\bar i}}{1+\l_i}+(n+\t\phi)^{1/(n-1)}\geq n(n-1)^{(1-n)/n}|\phi_i|^{2/n},
\]
and $\ve_0=\ve_0(\|\phi\|_{L^\infty}, \|F\|_{L^\infty}, n)$ is a fixed positive constant. Note that for any $\ve>0$, by Young's inequality, there exists a constant $C=C(\ve, n)$ such that
\begin{equation}\label{E-2-14}
|\nabla \phi|^2\leq \ve |\nabla \phi|^{2+2/n}+C(\ve, n).
\end{equation}
Now let \[u=\exp(-A(\phi))(|\nabla \phi|^2+1).\]
We can then compute, taking \eqref{E-2-11}, \eqref{E-2-13} and \eqref{E-2-14} into account, that
\begin{equation}\label{E-2-15}
\t_\phi u\geq \ve_0 |\nabla\phi|^{2+2/n}+e^{-A(\phi)}(n+\t\phi)-C|\nabla F||\nabla \phi|-C,
\end{equation}
where we have used that, by \eqref{E-2-5}, 
\[
\t_\phi (e^{-A(\phi)})\geq -C. 
\]
We then compute, for $p>0$, that
\begin{equation}\label{E-2-16}
\t_\phi (u^{p})=p u^{p-1}\t_\phi u+p(p-1)u^{p-2}|\nabla u|^2_\phi
\end{equation}
We  then compute, by \eqref{E-2-15},
\begin{equation*}
\int_M  u^{p-1}\t_\phi udvol_\phi\geq \int_M u^{p-1}(\ve_0|\nabla \phi|^{2+2/n}+e^{-A(\phi)}(n+\t\phi)-C|\nabla F||\nabla \phi|-C)dvol_\phi. 
\end{equation*}
It then follows, taking \eqref{E-2-16} into account, that
\begin{equation}\label{E-2-17}\begin{split}
&\int_M p(p-1)u^{p-2}|\nabla u|^2_\phi +p u^{p-1}(\ve_0|\nabla \phi|^{2+2/n}+e^{-A(\phi)}(n+\t\phi))\\
&\leq \int_M pu^{p-1}(C|\nabla F||\nabla \phi|+C))dvol_\phi. 
\end{split}
\end{equation}
When $p\geq 1$, we can compute that
\begin{equation}\label{E-2-18}
p(p-1)u^{p-2}|\nabla u|^2_\phi+pu^{p-1}e^{-A(\phi)}(n+\t\phi)\geq \ve p\sqrt{p-1} u^{p-3/2}|\nabla u|.
\end{equation}
Note that $u$ is bounded away from $0$, and 
$|\nabla \phi|\leq C u^{1/2}$. 
It then follows from \eqref{E-2-17}, \eqref{E-2-18} that
\begin{equation}\label{E-2-19}
\int_M \sqrt{p-1} u^{p-3/2}|\nabla u|dvol_g\leq C \int_M u^{p-1/2}(|\nabla F|+1)dvol_g.
\end{equation}
We can rewrite \eqref{E-2-19} as, for $p\geq 3/4$, that
\begin{equation}\label{E-2-20}
\int_M |\nabla (u^p)| dvol_g\leq C\sqrt{ p}\int_M u^p (|\nabla F|+1)dvol_g. 
\end{equation}
One can actually get the following,
\[
\int_M \left(|\nabla (u^p)|+C_1\sqrt{p} u^{p+1/2+1/n}\right) dvol_g\leq C\sqrt{ p}\int_M u^p (|\nabla F|+1)dvol_g,
\]
for some positive constant $C_1$; but we shall not need this. To get $L^\infty$ bound of $u$, we use the iteration method. Recall the following Sobolev inequality;
there exists a positive constant $c=c(M, g, n)$ such that for $f\in W^{1, 1}(M, g)$, we have
\begin{equation}\label{E-2-21}
\|f\|_{L^{\frac{2n}{2n-1}}}\leq c \int_M (|\nabla f|+Vol(M, g)^{-1/2n}|f|)dvol_g. 
\end{equation}
Taking $f=u^p$, it then follows from \eqref{E-2-20} and \eqref{E-2-21} that
\begin{equation}\label{E-2-22}
\|u^p\|_{L^{\frac{2n}{2n-1}}}\leq C\sqrt{p}\int_M u^p (|\nabla F|+1). 
\end{equation}
Now if $\|\nabla F\|_{L^{p_0}}$ is bounded, then by the H\"older inequality, we can get
\begin{equation}\label{E-2-23}
\int_M u^p (|\nabla F|+1)dvol_g\leq \left(\int_M u^{pq_0}dvol_g\right)^{1/q_0}\left(\int_M (|\nabla F|+1)^{p_0}dvol_g\right)^{1/p_0},
\end{equation}
where $1/p_0+1/q_0=1$.  So we can get that, from \eqref{E-2-22} and \eqref{E-2-23},
\begin{equation}\label{E-2-24}
\|u\|_{L^{\frac{2np}{2n-1}}}\leq (pC_2)^{1/2p}\|u\|_{L^{pq_0}},
\end{equation}
where $C_2=C_2(\|\phi\|_{L^\infty}, \|F\|_{W^{1, p_0}}, M, g, n)$. When $p_0>2n$, then $1<q_0<2n/(2n-1)$.
Let $b=q_0^{-1}2n/(2n-1)>1$, and take $p=q_0^{-1}b^k$ in \eqref{E-2-24} for $k\geq 0$, $k\in \N$,
it then follows  that
\begin{equation}\label{E-2-25}
\log \|u\|_{L^\infty}\leq \sum_{k=0}^\infty \frac{q_0}{b^k}\left(\log (q_0^{-1}b^k)+\log C_2\right) +\log \|u\|_{L^{1}}.
\end{equation}
It is clear that
\[
\|u\|_{L^1}\leq C\int_M |\nabla \phi|^2 dvol_g+C=-C\int_M\phi \t\phi dvol_g+C\leq C.
\]
It then follows that
\begin{equation}
\|u\|_{L^\infty}\leq C_3 =C_3(\|\phi\|_{L^\infty}, \|F\|_{W^{1, p_0}}, p_0, M, g, n). 
\end{equation}
\end{proof}

\begin{rmk}\label{R-2} Let $\l=1$ or $-1$ in \eqref{E-1}, then \eqref{E-2-15} still holds.  Hence Theorem \ref{T-1}  holds  for \eqref{E-1}.
\end{rmk}

\section{H\"older Estimates of Second Order and Solve the Equation}\label{S-2} To apply the method of continuity to solve \eqref{E-1-1}, one needs to obtain H\"older estimates of second order derivatives of $\phi$ and then higher order regularity follows from Schauder theory. For the Monge-Amp\`ere equations, $C^3$ estimates date back to Calabi's seminal third order estimates \cite{Calabi}; the idea is used in  \cite{Yau} to obtain $C^3$ estimates of $\phi$ for \eqref{E-1-1}.  Later on Evans \cite{Evans}, Krylov \cite{Krylov, Krylov1} proved that H\"older estimates of second order hold for fully nonlinear concave uniform elliptic operators.  All these results are originally stated for $F$ with  second derivatives or higher. But the H\"older estimates of second order derivatives are studied for uniform elliptic operators when right hand side has weaker regularity \cite{Schulz1, Caffarelli, Blocki, Wang}.  In particular, for the complex Monge-Amp\`ere equation,  Blocki \cite{Blocki} proved that  the H\"older estimates hold  when $F$ is  Lipschitz and $\t \phi$ is bounded.  With slight modifications of his argument, one can show that such an estimate  holds when $F\in W^{1, p_0}, p_0>2n$. 
The estimates can be  localized;  find a local potential $\Phi_0$ in a open domain $\Omega\subset M$ such that
 $g_{i\bar j}={\Phi_0}_{i\bar j}.$
We can rewrite \eqref{E-1-1} in $\Omega$ as 
\begin{equation}\label{E-3-1}
\det (v_{i\bar j})=f,
\end{equation}
where $v=\Phi_0+\phi, f=\exp(F) \det(g_{i\bar j})$. In particular, we can assume that $0<\l\leq \t v\leq \L$ and $f\in W^{1, p_0}$. Suppose $\|v\|_{L^\infty}, \|f\|_{W^{1, p_0}}\leq K$, then we have
\begin{lemma}\label{L-1} For any $\Omega^{'}\subset\subset \Omega$, 
\[
\|v\|_{C^{2, \alpha}(\Omega^{'})}\leq C(\Omega, \Omega^{'}, \l, \L, K),
\]
where $\alpha=\alpha(\Omega, \Omega^{'}, \l, \L, K)$. 
\end{lemma} 
\begin{proof}The result is proved in \cite{Blocki} (see Theorem 3.1 in \cite{Blocki}) when $F$ is Lipschitz. Our proof here is a slight modification. The idea is the same as in Evans-Krylov theory; but to deal with the case when $F$ has weaker regularity, one needs to use the Harnack inequality  as in  \cite{Gil-Tru} Theorem 8.18. 
 We use  notations as  in \cite{Blocki}. We observe that $(3.2)$ in \cite{Blocki} holds with $f^s\in L^{p_0}$ if $F\in W^{1, p_0}$. Hence the Harnack inequality (Theorem 8.18 in \cite{Gil-Tru})  still applies. By the Sobolev embedding, $f\in C^{\alpha}$ for $\alpha=1-2n/p_0$, then $(3.3)$ in \cite{Blocki} holds with $|z-w|$ replaced by $|z-w|^{\alpha/n}$. Then the argument as in (\cite{Gil-Tru} Section 17.4) or (\cite{Blocki} Theorem 3.1) applies. 
\end{proof}

Since $(M, g)$ is a smooth compact manifold,  a standard covering argument applies and one can get global H\"older estimate of second order derivatives of $\phi$ by using Lemma \ref{L-1}.
Now we are in the position to prove Theorem \ref{T}.
\begin{proof}If $F\in W^{1, p_0}$ on $M$ such that $\|F\|_{W^{1, p_0}}\leq K$ for some positive constant $K$.  Let $F^k$ be a sequence of smooth functions such that $F_k\rightarrow F$ in $W^{1, p_0}$; in particular,  We can assume $\|F_k\|_{W^{1, p_0}}\leq K+1$ for any $k$. 
By Yau's result \cite{Yau}, there is a smooth solution $\phi^k$ which solves
\begin{equation*}
\log\left(\frac{\det(g_{i\bar j}+\phi^k_{i\bar j})}{\det(g_{i\bar j})}\right)=F^k
\end{equation*}
such that $(g_{i\bar j}+\phi^k_{i\bar j})>0$ with normalized condition $\int_M\phi^k dg=0.$ Since $\|\phi^k\|_{L^\infty}$ is uniformly bounded \cite{K1} independent of $k$, we can get that  $|\nabla \phi^k|$, $|\t\phi^k|$ are both bounded by Theorem \ref{T-0} and Theorem \ref{T-1}, independent of $k$.
We can then get uniform H\"older estimate of second order by Lemma \ref{L-1}. To get $W^{3, p_0}$ estimate, we can localize the estimate as follows. Let $\p$ denote an arbitrary first order differential operator in a domain $\Omega\subset M$.  Once H\"older estimate of second order is proved,  we compute in $\Omega$
 \[
 \t_{\phi^k} (\p \phi^k)=\p( F^k-\log(\det(g_{i\bar j}))-g^{i\bar j}_{\phi^k}\p g_{i\bar j}.
 \]
 It then follows from $L^p$ theory, for example see \cite{Gil-Tru} Chapter 9,  for any $\Omega^{'}\subset\subset\Omega$, 
 \[
\|\p\phi^k\|_{W^{2, p_0}(\Omega^{'})}\leq C(\Omega, \Omega^{'}, p_0, K).
 \]
In particular we get, 
\begin{equation}\label{E-1-24}
\|\phi^k\|_{W^{3, p_0}(M, g)}\leq C(p_0, K).
\end{equation}
Then there is a subsequence of $\{\phi^k\}$ which converges to $\phi$ and $\phi$ is in $W^{3, p_0}$  such that $(g_{i\bar j}+\phi_{i\bar j})>0$ defines a $W^{1, p_0}$ (and $C^\alpha$,  $\alpha=1-2n/p_0$) K\"ahler metric; hence $\phi$ is a classical solution of \eqref{E-1-1}.   The solution of \eqref{E-1-1} is actually unique \cite{Calabi53, Yau}. 
\end{proof}

The Sobolev embedding theorem asserts that  $F\in C^{\alpha}, \alpha=1-2n/p_0$ when  $F\in W^{1, p_0}$; while $p_0>2n$ is exactly on the border line, namely when $p_0\leq 2n$, $F$ does not have to be in $L^\infty$. We might then  believe that the similar results  hold for $F\in C^\alpha$, namely
when $F\in  C^{\alpha}$, $\alpha\in (0, 1)$,  \eqref{E-1-1} has a classical solution $\phi\in C^{2, \alpha}$. 

The  gradient estimate and estimate of $\t \phi$ in the present paper are both global; actually such estimates cannot be made purely local, for example see \cite{Blocki99, He}. 
But one might ask whether such an estimate holds or not for the Dirichlet problem of the complex Monge-Amp\`ere equations in a bounded strongly pseudoconvex domain in $\C^n$. In particular, one might  consider the complex Monge-Amp\`ere equation in a bounded strongly pseudoconvex domain $\Omega\subset \C^n$,  
\[
\log \det(\phi_{i\bar j})=f,
\]
where $f$ is Lipschitz (or $C^\alpha$), and $\phi|_{\p \Omega}=0$; is $\phi\in C^{2, \alpha}$?
This problem was studied in \cite{Schulz} when $f$ is Lipschitz; however the proof in \cite{Schulz} has a gap \cite{Blocki03, He}.
In essence, one would like to understand if the results for the real Monge-Amp\`ere equation ( see Caffarelli \cite{Caf}, Trudinger-Wang \cite{TruWang}  for example) hold or not for the complex Monge-Amp\`ere equation.

\end{document}